\documentclass[11pt,reqno,a4paper]{amsart}

\usepackage{amsmath}
\usepackage{amssymb}
\usepackage{graphicx}

\newtheorem{thm}{Theorem}[section]
\newtheorem{prop}[thm]{Proposition}

\newtheorem{cor}[thm]{Corollary}

\theoremstyle{remark}
\newtheorem{rema}{Remark}[section]

\def\mb{\mathbf}

\def\mr{\mathrm}
\def\mc{\mathcal}

\newlength{\equwidth}
\newcommand{\crd}{\mbox{$                                     
\begin{picture}(9,8)(1.6,0.15)
\put(1,0.2){\mbox{$ D \hspace{-7.8pt} /$}}
\end{picture}$}}

\DeclareMathOperator{\Ric}{Ric}

\DeclareMathOperator{\PP}{P}

\def\ga{\gamma}
\def\de{\delta}

\def\ka{\kappa}

\def\rh{\rho}

\def\si{\sigma}

\def\Ups{\Upsilon}
\def\ph{\varphi}

\def\om{\omega}
\def\Ga{\Gamma}
\def\De{\Delta}
\def\Th{\Theta}
\def\La{\Lambda}

\def\Om{\Omega}

\def\pa{\partial}
\def\t{\otimes}

\def\delstar{\pa^*}

\def\goesto{\rightarrow}

\def\equiv{\Leftrightarrow}
\def\embed{\hookrightarrow}

\def\cinf{\ensuremath{\mathrm{C}^\infty}}

\def\na{\mathrm{\nabla}}

\def\lapl{\mathrm{\square}}
\def\triang{\mathrm{\bigtriangleup}}
\def\Lapl{\mathrm{\triang}}

\def\im{\mathrm{im}\ }

\def\G{\mathcal{G}}

\DeclareMathOperator{\GL}{GL}
\DeclareMathOperator{\Spin}{Spin}
\DeclareMathOperator{\CSpin}{CSpin}

\def\CO{\mathrm{CO}}

\def\ti{\tilde}

\def\rr{\ensuremath{\mathbb{R}}}

\def\HH{\ensuremath{\mathcal{H}}}

\def\VV{\ensuremath{\mathcal{V}}}

\def\ST{\ensuremath{\mathcal{S}}}

\def\SU{\ensuremath{\mathrm{SU}}}

\def\SL{\ensuremath{\mathrm{SL}}}

\def\SO{\ensuremath{\mathrm{SO}}}

\def\GL{\ensuremath{\mathrm{GL}}}

\def\sl{\ensuremath{\mathfrak{sl}}}
\def\so{\ensuremath{\mathfrak{so}}}

\def\g{\ensuremath{\mathfrak{g}}}

\def\p{\ensuremath{\mathfrak{p}}}
\def\q{\ensuremath{\mathfrak{q}}}

\newcommand{\bg}{\mbox{\boldmath{$ g$}}}

\DeclareMathOperator{\Hol}{Hol}

\def\today{\ifcase\month\or
 January\or February\or March\or April\or May\or June\or
 July\or August\or September\or October\or November\or December\fi
 \space\number\day, \number\year}

 \usepackage{amssymb}
\usepackage{paralist} \usepackage{graphicx} \usepackage{hyperref}
\usepackage[all]{xy}

\def\mb{\mathbf} 

\def\mr{\mathrm}
\def\mc{\mathcal}

\def\calG{\ensuremath{\mathcal{G}}}
\def\ef{\ensuremath{\mathrm{f}}}
\def\ee{\ensuremath{\mathrm{e}}}

\def\t{\otimes}
\title[A non-normal Fefferman-type construction]{A non-normal Fefferman-type construction of split-signature conformal structures admitting twistor spinors}
\author{Matthias Hammerl and Katja Sagerschnig} \email{matthias.hammerl@univie.ac.at,
katja.sagerschnig@univie.ac.at} \address{Faculty of Mathematics,
University of Vienna, Nordbergstra\ss e~15, A--1090 Wien, Austria}
\date{\today}


\begin{document}

\maketitle

\begin{abstract}
  We treat a non-normal Fefferman-type construction based on
  an inclusion $\SL(n+1)\embed\Spin(n+1,n+1)$. The construction associates
  a split signature $(n,n)$-conformal spin structure to a projective
  structure of dimension $n$. For $n\geq 3$ the induced
  conformal Cartan connection is shown to be normal if and only if it is flat.
  The main technical work of this article
  consists in showing that in the non-flat case 
  the normalised conformal Cartan connection
  still allows a parallel (pure) spin-tractor and thus a corresponding
  (pure) twistor spinor on the conformal space. The Fefferman-type
  construction presented here is an alternative approach to
  study a construction of Dunajski-Tod.
\end{abstract}

\section{Introduction}

The original Fefferman construction \cite{fefferman} canonically
associated a conformal structure on a circle bundle over a CR-structure.
The resulting conformal structure is rather special: it admits
solutions to certain invariant overdetermined equations, in particular,
it carries a light-like conformal Killing field. In fact, it was
shown by Sparling, cf. \cite{graham-sparlingchar}, that a conformal structure is
the Fefferman-space of some CR-structure if and only if it admits
such a Killing field which also satisfies additional (conformally invariant)
properties. This yields the characterisation of the CR-Fefferman spaces.
The characterising property can alternatively be understood as a \emph{holonomy reduction} of the conformal structure:
It was shown in \cite{cap-gover-cr} that a conformal structure $(M,\mc{C})$\
is locally the Fefferman-space of a CR-structure if and only if its
conformal holonomy satisfies $\Hol(\mc{C})\subset\SU(p+1,q+1)\subset\SO(2p+2,2q+2).$ 

A  generalisation of the original Fefferman-construction was described in \cite{cap-correspondence}, and in recent years
a number of constructions have been discussed in that framework:
The original construction was treated via this approach in \cite{cap-gover-cr}, \cite{mrh-sag-rank2} discussed Nurowski's conformal structures \cite{nurowski-metric} that are associated to generic rank two distributions on $5$-manifolds, \cite{jesse-quaternionic} treated a Fefferman-type construction of
conformal structures from quaternionic contact structures,  \cite{mrh-sag-twistors} discussed Bryant's \cite{bryant-3planes} conformal structures associated with generic $3$-planes on $6$-manifolds.

In all cited cases the Fefferman-type construction is \emph{normal}: this says that, starting from the normal Cartan connection encoding the original geometric structure (e.g., a CR-structure, a generic distribution, a quaternionic contact structure) the induced conformal Cartan connection form that is built via the Fefferman-type construction is again normal. 
This immediately implies that the holonomy of the conformal structure reduces to the included subgroup and makes it possible to derive a holonomy-based
characterisation of the induced structures. 

In this paper we discuss a non-normal Fefferman-type construction. We
associate a split signature $(n,n)$ conformal spin structure to a projective
structure of dimension $n$.  The construction is based
on an inclusion $\SL(n+1)\embed\Spin(n+1,n+1)$. If $n=2$ this construction is shown to be normal, and the usual consequences on conformal  holonomy reduction, Proposition \ref{prop-22}, and symmetry-decomposition, \ref{prop-22kill}, can be derived. 
In addition, it is also possible in this case to understand the space of (almost) Einstein metrics in the induced conformal class in terms of projective data, Proposition \ref{prop-22ein}.
For $n\geq 3$,  the induced conformal Cartan connection is shown to be normal if and only if the original projective structure was already flat, Proposition \ref{prop-nonnormal}. This fact immediately poses problems for the goal of relating the original projective and the induced conformal geometric structure: since the induced conformal Cartan connection form is not normal, its curvature and holonomy are no well defined conformally invariant objects. To obtain information on the conformal
structure it is thus necessary to understand how the normal conformal connection differs from this one. 
We derive strong restrictions on the form of the normalised Cartan connection in Proposition \ref{prop-delstarformnn}. These imply in particular that the induced conformal structures, which carry a canonical spin structure, are endowed
with a solution of the twistor spinor equation, Theorem \ref{prop-norm2}.

The original motivation for this Fefferman-type construction comes from two sources. The fist one is work by Dunajski-Tod, \cite{dunajski-tod}:
Extending a construction due to Walker \cite{walker}, which associates a pseudo-Riemannian split signature $(n,n)$-metric to an affine torsion-free connection on an $n$-manifold, they associate a conformal split signature $(n,n)$-metric to a projective class of torsion-free affine connections on an $n$-manifold. Using a normal form
for the induced metrics it is also shown that they admit a twistor spinor. 
This construction is also discussed  in Dunajski-West, \cite{dunajski-west}. The second source is a paper by P. Nurowski and G. Sparling, \cite{nurowski-sparling-cr}, which treats the construction from $2$-dimensional projective structures
to  conformal structures of signature $(2,2)$ using Cartan connections. A generalisation of this approach to higher dimensions was mentioned in \cite{nurowski-projective-metric}.
The precise relation between  the cited works and the construction here has been shown recently by \v{S}ilhan-{\v{Z}}{\'a}dn{\'{\i}}k, \cite{silhan-zadnik-esitalk}: 
It is based on an interpretation of the explicit formula for the Dunajski-Tod conformal metric in terms of 'Thomas's projective parameters', which in turn has relations to tractor calculus for projective structures and the projective ambient metric, \cite{thomass}, and thus provides a link to the Fefferman-type interpretation of the construction.

\subsubsection*{Outlook}
This constructions leads to interesting questions for future work.
In signature $(2,2)$ Dunajski-Tod could show, \cite{dunajski-tod}, Theorem 4.1, that one has a $1:1$-correspondence between compatible
(pseudo)-Riemannian metrics for the original projective class and (para)-K\"{a}hler-metrics in the induced conformal class. In forthcoming joint work
with J. \v{S}ilhan and V. {\v{Z}}{\'a}dn{\'{\i}}k  we will discuss this relation in
terms of BGG-solutions to certain projective and conformal equations.
Another problem that will be treated is to characterise the resulting conformal structures. As is shown in this article, the existence of a certain pure twistor spinor should play a large role in this, but additional data is necessary to characterise
the structures precisely.
It would also be interesting to study the ambient-metrics of the
induced conformal Fefferman-spaces, as it was done for certain generic $2$-distributions in \cite{leistner-nurowski-g2ambient}.

\subsubsection*{Acknowledgements}
The first author has enjoyed discussions with Maciej Dunajski at
the workshop 'Dirac operators and special Geometries' held in 2009 in Rauischholzhausen.
Both authors benefited from discussions with Josef \v{S}ilhan and Vojtech {\v{Z}}{\'a}dn{\'{\i}}k - in particular we are very thankful for their comments and suggestions on a draft of this paper.\\
M. H. is supported by the project P23244-N13
of the "Fonds zur F\"{o}rderung der wissenschaftlichen For\-schung"
(FWF).
K. S. is supported by an Erwin Schr\"{o}dinger-Fellowship, J 3071-N13 (FWF).

\section{Basic facts about parabolic geometries and some background on projective and conformal structures}
\subsection{Parabolic geometries.}
Let $G$ be a Lie group with Lie algebra $\g$ and $P\subset G$ a closed
subgroup with Lie algebra $\p$. A \emph{Cartan geometry} $(\mathcal{G},\omega)$
of type $(G,P)$  is a $P$-principal bundle $\mathcal{G}\to M$ together
with a \emph{Cartan connection} $\omega\in\Omega^1(\mathcal{G},\g)$, i.e.,
a $\g$-valued $1$-form on $\mathcal{G}$ that i) is $P$-equivariant, ii)
maps each fundamental vector field $\zeta_{X}$  to its generator
$X\in\p$, and iii) defines a linear isomorphism
$\omega(u):T_u\mathcal{G}\to\mathfrak{g}$ for each $u\in\mathcal{G}$.

The \emph{curvature} of a Cartan connection $\omega$ is the $2$-form
$K\in\Omega^2(\mathcal{G},\g)$ defined as
\begin{align*}
K(\xi,\eta)=d\omega(\xi,\eta)+[\omega(\xi),\omega(\eta)]
\end{align*}
for $\xi,\eta\in\mathfrak{X}(\mathcal{G})$. It is equivalently
encoded in the curvature function
$\kappa:\mathcal{G}\to\Lambda^2(\g/\p)^*\otimes \g$
\begin{align*}
\kappa(u)(X+\p,Y+\p)=K(\omega^{-1}(u)(X),\omega^{-1}(u)(Y)).
\end{align*}
The curvature is a complete obstruction to local equivalence with the homogeneous model $G\to G/P$ endowed with the Maurer-Cartan form $\omega^{MC}$.
If the image of $\kappa$ is contained in $\Lambda^2(\g/\p)^*\otimes \p,$ then $(\mathcal{G},\omega)$ is called \emph{torsion-free}.

A \emph{parabolic geometry} is a Cartan geometry of type $(G,P)$, where
$G$ is a semisimple Lie group and $P$ is a parabolic subgroup. Every
parabolic subgroup is the semidirect product $P=G_0\ltimes P_{+}$ of a
reductive Lie group $G_0$ and a normal subgroup $P_{+}\subset P$. The Lie
algebra $\p_{+}$  is the orthogonal complement of $\p$ in $\g$ with
respect to the Killing form, $P_{+}=\mathrm{exp}(\p_{+})$ and $G_0\cong
P/P_{+}$. Since $G_0$ is reductive,  its Lie algebra
$\g_0=\g_0^{ss}\oplus \mathfrak{z}(\g_0)$ decomposes into the semisimple
part $\g_0^{ss}=[\g_0,\g_0]$ and the centre $\mathfrak{z}(\g_0)$.
For parabolic geometries there is a natural choice of a normalisation
condition, which reads $\partial^*(\kappa)=0$, where
\begin{align*}\partial^*:\Lambda^k(\g/\p)^*\otimes \g\to\Lambda^{k-1}(\g/\p)^*\otimes\g\end{align*} is the
\emph{Kostant codifferential}  \cite{kostant-61}.  The \emph{harmonic curvature} $\kappa_{H}$ of a normal parabolic geometry is the image of $\kappa$ under the projection $\mathrm{ker} \partial^*\to\mathrm{ker}\partial/\mathrm{im}\partial^*$. 
The parabolic geometries we are mainly interested here (i.e. projective and conformal geometries) are automatically \emph{regular}, see \cite{cap-slovak-par}, and in that case the entire curvature $\kappa$ is completely determined by  $\kappa_{H}$.

A technical tool that we will often employ are Weyl structures for
parabolic geometries, cf. \cite{cap-slovak-weyl,cap-slovak-par} for a detailed account.
 A \emph{Weyl structure}\ of $(\G,\om)$ is a reduction of structure group $  j:\G_0 \embed\G$ of the $P$-principal bundle
  $\G$\ to a $G_0$-bundle $\G_0$.

Every Cartan connection $\omega$ naturally extends to a principal bundle
connection $\hat{\omega}$ on the $G$-principal bundle
$\hat{\mathcal{G}}=\mathcal{G}\times_{P}G$. The principal bundle
connection $\hat{\omega}$ induces a vector bundle connection
$\nabla^{\mathcal{V}}$ on each associated bundle
$\mathcal{V}=\mathcal{G}\times_{P}\mathbb{V}=\hat{\mathcal{G}}\times_{G}\mathbb{V}$
for a $G$-representation $\mathbb{V}$. Bundles $\mathcal{V}$ and
connections $\nabla^{\mathcal{V}}$ arising in this way are called
\emph{tractor bundles} and \emph{tractor connections}. The 
tractor connections induced by  normal Cartan connections
 for parabolic geometries
are called normal tractor connections.

\subsection{Normal solutions of first BGG-equations as parallel tractor sections}\label{sec-bgg}
In \cite{BGG-2001}, and later in a simplified manner in \cite{BGG-Calderbank-Diemer},\ it was shown that for a given tractor bundle $\mc{V}$ one can associate a natural sequence of differential operators,
\begin{align*}
  \HH_0\overset{{\Th_0^{\mc{V}}}}{\goesto}\HH_1\overset{{\Th_1^{\mc{V}}}}{\goesto}\cdots\overset{{\Th_{n-1}^{\mc{V}}}}{\goesto}
  \HH_n.
\end{align*}
The operators ${\Th_k^{\mc{V}}}$\ are the \emph{BGG-operators}, which operate between natural
sub-quotients $\HH_k$\ of $\Om^k(M,\VV)$.
We remark that ${\Th_k^{\mc{V}}}$\ form a complex if and only if the geometry $(\G,\om)$\ is locally flat.

We won't discuss the general construction here,
for which we refer to the articles mentioned above or \cite{mrh-thesis},
and just state the basic properties of the \emph{first BGG-operator} ${\Th_0^{\mc{V}}}:\Ga(\HH_0)\goesto\Ga(\HH_1)$.
The operator defines an overdetermined system of differential equations on $\si\in\Ga(\HH_0)$,
${\Th_0^{\mc{V}}}(\si)\overset{!}{=}0$, which is termed the \emph{first BGG-equation}.

For the projective and conformal structures we discuss below, we will be able to
encode a number of interesting geometric equations as first BGG-equations.
In those cases solutions of the first BGG-equations are always in $1:1$-correspondence
with parallel sections of the defining tractor bundle $\mc{V}$, cf. \cite{mrh-thesis}.
In general one only has $1:1$-correspondence between parallel sections and a subspace
of solutions of ${\Th_0^{\mc{V}}}(\si)=0$, which are called \emph{normal solutions}. 
This correspondence is realised as follows:
The bundle $\HH_0$\ is a natural quotient of $\mc{V}$,
$\mc{V}\overset{\Pi_0}{\goesto}\HH_0,$
and the BGG-construction defines a natural differential splitting operator
$\Ga(\mc{H}_0)\overset{{L_0^{\mc{V}}}}{\goesto}\Ga(\mc{V})$ of that projection.
Then a solution of ${\Th_0^{\mc{V}}}(\si)=0$\ is normal if and only if
$\na^{\mc{V}}{L^{\mc{V}}}(\si)=0$.

In the following we describe projective and conformal structures.
To write down explicit formulas it will be useful
to employ abstract index
notation, cf. \cite{penrose-rindler-87}: we write $\mb{E}_a=T^*M,\mb{E}^a=TM$ and
multiple indices as in $\mb{E}_{ab}=T^*M\t T^*M$\ denote tensor products.
Indices between squared brackets are skew, as in $\mb{E}_{[ab]}=\La^2 T^*M$,
and indices between round brackets are symmetric, as in $\mb{E}^{(ab)}=S^2 TM$.
\subsection{Projective Structures}\label{section-projective}

Let $M$\ be a manifold of dimension $n\geq 2$\ endowed with a projective class of 
torsion-free affine connections $[D]$: two connections $D$\ and $\hat D$\ are
projectively equivalent if they describe the same geodesics as unparameterised
curves. This is the case if and only if there is a $\Ups_a\in\mb{E}_a$\ such
that for all $\xi^a\in\mb{E}^a$,
\begin{align}\label{projectivehatD}
  \hat D_a\xi^b=D_a\xi^b+\Ups_a\xi^b+\Ups_p\xi^p \de_a^{\; b},
\end{align}
where $\de=\mr{id}_{TM}$\ is the Kronecker-symbol for the identity on $TM$,
cf. e.g. \cite{eastwood-matveev}\ and \cite{thomass}.

Let $R$\ be the curvature of $D$. With the Schouten tensor $\mr{P}\in\mb{E}_{(ab)}$,
\begin{align}\label{def-projectiveschouten}
  \mr{P}_{ab}=\frac{1}{n-1}R_{pa\;\; b}^{\;\;\ p}
\end{align}
one has the projective Weyl- and Cotton tensor
\begin{align}\label{formulaCprojective}
  C_{c_1c_2\; p}^{\quad\ a}&=R_{c_1c_2\; p}^{\quad\; a}+\mr{P}_{c_1p}\de_{c_2}^a-\mr{P}_{c_2p}\de_{c_1}^a,\\ \label{formulaAprojective}
  A_{ac_1c_2}&=2D_{[c_1}\mr{P}_{c_2]a}.
\end{align}

An oriented projective structure $(M,[D])$ is equivalently encoded in
a normal parabolic geometry of type
$(SL(n+1),P)$, where $P$ is the stabiliser of a ray in the
standard representation $\mathbb{R}^{n+1}$. This classical result
goes back to \'{E}.Cartan, \cite{cartan-projective}. For a modern
treatment we refer to \cite{sharpe,cap-slovak-par}.

The parabolic subgroup $P\subset G$ is a semidirect product
$P=\GL(n)\ltimes (\rr^n)^*$.
The $1$-dimensional representation of $P$
\begin{align*}
\GL(n)\ltimes(\rr^n)^*&\goesto \rr_+,\ 
(C,X)\mapsto \det(C)^{w\frac{n}{n+1}}
\end{align*}
is denoted by $\rr[w]$: the associated space
$\mb{E}[w]:=\G\times_P \rr[w]$\ are projective $w$-densities,
which are just usual densities with a suitable parametrisation.
If $V$\ is a $P$-representation and $\mc{V}=\G\times_P V$ its associated
bundle we will simply write $V[w]=V\t \rr[w]$ resp. $\mc{V}[w]=\mc{V}\t\mb{E}[w]$\
for the weighted versions of the modelling representation resp. the corresponding
associated bundles.

 For the projective geometry $(M,[D])$ any choice
of affine connection $D\in[D]$ yields a projective Weyl structure, and
in particular the structure group of any tractor bundle is reduced to $G_0=\SL(n)$.

\subsubsection{The projective standard tractor bundle}
This is the associated bundle $\mc{T}=\G\times_P\rr^{n+1}$.
With respect to a choice of $D\in[D]$ we have
$[\mc{T}]_{D}=
  \begin{pmatrix}
    \mb{E}[-1] \\
    \mb{E}^a[-1] 
  \end{pmatrix}$,
and $\Pi_0:\mc{T}\goesto \mb{E}^a[-1]=\HH_0^{\mc{T}}$ is the projectively invariant
projection to the lowest slot.
The tractor connection is given by
$  \na^{\mc{T}}_c
  \begin{pmatrix}
    \rh \\
    \si^a
  \end{pmatrix}
  =
  \begin{pmatrix}
    D_c\rh-\mr{P}_{cp}\si^p \\
    D_c\si^a+\rh \de_c^{\; a}
  \end{pmatrix}.
$
The BGG-splitting operator is
\begin{align*}
  L_0^{\mc{T}}:\mb{E}^a[-1]\goesto\ST,\ 
  \si^a\mapsto
  \begin{pmatrix} 
    -\frac{1}{n}D_p\si^p\\
    \si^a 
  \end{pmatrix}
\end{align*}
and the first BGG-operator of $\mc{T}$\ is
\begin{align}\label{proj-standardBGG}
  \Th_0^{\mc{T}}:\mb{E}^a[-1]\goesto{\mb{E}_0}_c^{\; a}[-1],\
  \si^a\mapsto D_c\si^a-\frac{1}{n}\de_c^{\; a}D_p\si^p.
\end{align}
Thus, $\ker\Th_0^{\mc{T}}$\ consists of vector fields
which are mapped to multiples of the identity by $D$.

\subsubsection{The projective dual standard tractor bundle}
The dual bundle to $\mc{T}$\ is $\mc{T}^*=\G\times_P{\rr^{n+1}}^*$.
Its decomposition under $D\in[D]$ is
$[\mc{T}^*]_{D}=
  \begin{pmatrix}
    \mb{E}_a[1] \\
    \mb{E}[1] 
  \end{pmatrix}
$
and $\Pi_0:\mc{\mc{T}}^*\goesto \mb{E}[1]=\HH_0^{\mc{T}^*}$ is the projectively invariant
projection to the lowest slot.
The tractor connection is 
$ \na^{\mc{T}^*}_c
  \begin{pmatrix}
    \ph_a \\
    \si
  \end{pmatrix}
  =
  \begin{pmatrix}
    D_c\ph_a +\mr{P}_{ca}\si \\
    D_c\si-\ph_c
  \end{pmatrix}.
$
The first splitting operator of $\mc{T}^*$ is
\begin{align*}
  L_0^{\mc{T}^*}:\mb{E}[1]\goesto\mc{T}^*,\ 
  \si\mapsto
  \begin{pmatrix}
    D_a\si \\
    \si
  \end{pmatrix}
\end{align*}
and the first BGG-operator is
\begin{align}\label{aRfequ}
  \Th_0^{\mc{T}^*}:\mb{E}[1]\goesto\mb{E}_{(ab)}[1],\ 
  \si \mapsto D_aD_b \si+\si \mr{P}_{ab}.
\end{align}

Let $\si\in\cinf(M)$\ be a solution of $\Th_0^{\mc{T}^*}(\si)=0$
and define $\Ups_a=D_a(\log \frac{1}{\lvert{\bar{\si}}\rvert})$.
Then $\Ups$\ is a well-defined $1$-form on $U:=M\backslash \si^{-1}(\{0\})$,
and one can form the connection $\hat D$,\eqref{projectivehatD}, that
is projectively equivalent to the restriction of $D$\ to $U$.
Then \eqref{aRfequ}\ implies, cf. \cite{mrh-thesis,cgh-projective},
that $\Ric(\hat D)=0$.
We will call any connection $\hat D$\ 
that is defined on an open-dense subset $U$ of $M$
and is contained in the 
restriction of $[D]$ to $U$ an \emph{almost Ricci-flat structure} 
of $[D]$, or $\hat D\in\mb{aRs}([D])$. It will sometimes
be useful to regard $\mb{aRs}([D])\subset[D]$, even though the
almost Ricci-flat structures only give connections on an open-dense
subset of $M$.
Then, cf. \cite{mrh-thesis,cgh-projective},
\begin{align}
  \label{aRsref}
  \mb{aRs}([D])\cong\ker \Th_0^{\mc{T}^*}.
\end{align}

\subsection{Conformal spin structures}\label{section-conformalspin}

A\emph{ conformal\ structure} of
signature $(n,n)$ on an $n=p+q$-dimensional
manifold $M$ is an equivalence class $\mc{C}$ of
pseudo-Riemannian metrics with two metrics $g$ and $\hat{g}$ being
equivalent if $\hat{g}=\mr{e}^{2f}g$ for a function $f\in
C^{\infty}(M)$.  Suppose we have a manifold with a conformal structure
of signature $(n,n)$.  Let $\mathcal{G}_0$ be the associated conformal
frame bundle with structure group the conformal group $\CO_o(n,n)=\rr_+\times \SO_o(n,n)$
preserving both orientations.  Then a \emph{conformal spin structure}
on $M$ is a reduction of structure group of $\mathcal{G}_0$ to
$\CSpin(n,n)=\mathbb{R}_{+} \times \Spin(n,n)$.
As for projective structures, it is useful to employ a suitable parametrisation
of densities: the \emph{conformal density bundles} $\mb{E}[w]$,
which are the line
bundles associated to the $1$-dimensional representations $(c,C)\mapsto
c^r\in\rr_+$\ of $\CSpin(n,n)=\rr_+\times \Spin(n,n)$.  

Let us now briefly introduce the main curvature quantities of the conformal
structure $\mc{C}$, cf. e.g. \cite{eastwood-notes-conformal}.
For $g\in\mc{C}$, let, with $m=2n$,
\begin{align*}
  P=\mr{P}(g):=\frac{1}{m-2}(\mr{Ric}(g)-\frac{\mr{Sc}(g)}{2(m-1)}g)
\end{align*}
be the \emph{Schouten\ tensor}; this is a trace modification
of the Ricci curvature $\mr{Ric}(g)$\ by a multiple
of the scalar curvature $\mr{Sc}(g)$. The trace of
the Schouten tensor is denoted $J=g^{pq}\PP_{pq}$.

It is well known that (since we always have dimension $\geq 4>3$), the complete obstruction against conformal flatness of $(M,\mc{C})$\ is the \emph{Weyl curvature}
\begin{align*}
  C_{ab\; d}^{\;\;\; c}:=R_{ab\; d}^{\;\;\; c}-2\de_{[a}^c\PP_{b]d}+2\bg_{d[a}\PP_{b]}^{\; c},
\end{align*}
where indices between square brackets are skewed over.

A conformal spin structures of signature $(n,n)$
is equivalently encoded in a normal parabolic
geometry of type $(Spin(n+1,n+1),\tilde{P}),$ where $\tilde{P}$ is the
stabiliser of a ray in $\mathbb{R}^{n+1,n+1}$.
Any choice of $g\in\mc{C}$ yields a Weyl structure of $(\G,\om)$,
and this reduces the structure group of a tractor bundle to $\ti G_0=\Spin(n,n)$.

\subsubsection{The conformal standard tractor bundle}\label{secstd}
This is the associated bundle $\ti{\mc{T}}=\ti\G\times_{\ti P}\rr^{n+1,n+1}$,
and with respect to $g\in\mc{C}$ it decomposes
$  [\ti{\mc{T}}]_g=
  \begin{pmatrix}
    \mb{E}[-1]\\
    \mb{E}_a[1] \\
    \mb{E}[1]
  \end{pmatrix},
$
and $\Pi_0:\ti{\mc{T}}\goesto \mb{E}[1]=\HH_0^{\ti{\mc{T}}}$ is the projectively invariant
projection to the lowest slot.
$\ti{\mc{T}}$ carries invariant tractor metric
$
  [\bf{h}]_g=
  \begin{pmatrix}
    0 & 0 & 1 \\
    0 & \bg & 0 \\
    1 & 0 & 0
  \end{pmatrix},
$
which is compatible with the standard tractor connection
$  [\na^{\ti{\mc{T}}}_c 
  \begin{pmatrix}
    \rh \\
    \ph_a \\
    \si
  \end{pmatrix}]_g
  =
  \begin{pmatrix}
    D_c \rh-\mr{P}_{c}^{\; b}\ph_b \\
    D_c\ph_a+\si \mr{P}_{ca}+\rh \bg_{c a}\\
    D_c \si-\ph_c
  \end{pmatrix}.
$
The BGG-splitting operator of $\ti{\mc{T}}$\ is
\begin{align}\label{splitStd}
  L_0^{\ti{\mc{T}}}:\mb{E}[1]\goesto \ti{\mc{T}},\
  \si\mapsto
  \begin{pmatrix}
    \frac{1}{2n}(\Lapl-J)\si \\
    D\si \\
    \si
  \end{pmatrix}
\end{align}
with the convention $\Lapl=-D^pD_p$.
The first BGG-operator is 
\begin{align}\label{EinsteinBGG}
  \Th_0^{\ti{\mc{T}}}:\mb{E}[1]\goesto{\mb{E}_0}_{(ab)}, \
  \si\mapsto (D_aD_b\si+\mr{P}_{ab}\si)_0.
\end{align}
It is well known that
\begin{align}\label{aEs-equ}
  (D_aD_b\si+\mr{P}_{ab}\si)_0=0 \ \equiv \ \si^{-2}\bg\ \mr{is\ Einstein\ on}\ U,
\end{align}
and we call the set of solutions of \eqref{aEs-equ}
the space of \emph{almost\ Einstein\ structures}\ of $\mc{C}$, cf. \cite{gover-aes},
i.e.:
\begin{align}\label{def-aEs}
  \mb{aEs}(\mc{C})=\ker \Th_0^{\ti{\mc{T}}}\subset\mb{E}[1].
\end{align}
It will sometimes be convenient to regard $\mb{aEs}(\mc{C})\subset\mc{C}$,
even if these Einstein-metrics are only defined on an open-dense subset.

\subsubsection{The spin tractor bundle}\label{sec-spintrac}
Since $\mc{C}$\ is a conformal spin structure
and modelled on a Cartan geometry of type $(\Spin(n+1,n+1),\ti P)$\
we can define the spin tractor bundle as $\ti{\mc{S}}=\ti\G\times_{\ti P}\De^{n+1,n+1}$.
Since we work in even signature, this decomposes into $\ti{\mc{S}}_{\pm}=\ti\G\times_{\ti P}\De_{\pm}^{n+1,n+1}$.
Under a choice of $g\in\mc{C}$\ the spin tractor bundles decompose as follows:
$
  [\ti{\mc{S}}_{\pm}]_g=
  \begin{pmatrix}
    S_{\mp}[-\frac{1}{2}] \\
    S_{\pm}[\frac{1}{2}]
  \end{pmatrix}.
$
$\Pi_0:\ti{\mc{S}}_{\pm}\goesto S_{\pm}[\frac{1}{2}]=\HH_0^{\ti{\mc{S}}_{\pm}}$ is the projectively invariant
projection to the lowest slot.
The Clifford action of the conformal standard tractor bundle $\ti{\mc{T}}$\ on $\ti{\mc{S}}$
is given by
\begin{align}\label{traCli}
    \begin{pmatrix} \rh \\ \ph_a \\ \si
  \end{pmatrix} \cdot
  \begin{pmatrix} \tau \\ \chi
    \end{pmatrix}
    =
  \begin{pmatrix} -\ph_a\cdot\tau+\sqrt{2}\rh \chi \\
\ph_a\cdot\chi-\sqrt{2}\si \tau
  \end{pmatrix},
\end{align}
cf. \cite{mrh-thesis,mrh-coupling}.
$\ti{\mc{S}}=\ti{\mc{S}}_{+}\oplus\ti{\mc{S}}_-$\ carries the spin tractor connections that is induced from
the standard tractor connection on $\ti{\mc{T}}$:
$  [\na^{\ti{\mc{S}}}_c
  \begin{pmatrix}
    \tau \\
    \chi
  \end{pmatrix}
  ]_g
  =
  \begin{pmatrix}
    D_c\tau+\frac{1}{\sqrt{2}}\PP_{cp}\ga^p\chi\\
    D_c\chi+\frac{1}{\sqrt{2}}\ga_c\tau
  \end{pmatrix}.
$

The BGG-splitting operator of $\ti{\mc{S}}_{\pm}$\ is
\begin{align}\label{twisplit} 
  L_0^{\ti{\mc{S}}_{\pm}}:\Ga(S_{\pm}[\frac{1}{2}])&\goesto \Ga(\ti{\mc{S}}_{\pm}), \
   \chi\mapsto
  \begin{pmatrix} 
    \frac{1}{\sqrt{2}n}\crd\chi \\ 
    \chi
  \end{pmatrix}.
\end{align}
Here
\begin{align*}
  &\crd:\Ga(S_{\pm})\goesto \Ga(S_{\mp}),\ 
  \crd:=\ga^pD_p,
\end{align*}
is the \emph{Dirac} operator.
The first BGG-operator is
\begin{align*}
  &\Th^{\ti{\mc{S}}}_0:\Ga(S_{\pm}[\frac{1}{2}])\goesto \Ga(T^*M\t S_{\pm}[\frac{1}{2}]),\\
  &\Th^{\ti{\mc{S}}}_0(\chi):=D\chi+\frac{1}{2n}\gamma\crd\chi.
\end{align*}
This is the \emph{twistor operator} (cf. e.g. \cite{baum-friedrich-twistors}),
which is alternatively described as the projection of the
Levi-Civita derivative of a spinor to the kernel of Clifford multiplication. The kernel of the twistor operator is called the space of twistor spinors $\mb{Tw}(\mc{C})$, and  
$\Pi_0$\ induces an isomorphism of the space of $\na^{\ti{\mc{S}}}$-parallel
sections of $\ti{\mc{S}}$\ with  $\mb{Tw}(\mc{C})$\ in
$\Ga(S[\frac{1}{2}])$.

\subsubsection{Conformal holonomy}
The \emph{conformal holonomy} of a conformal spin structure $\mc{C}$ is defined as 
  \begin{align}\label{defhol}
\Hol(\mc{C}):=\Hol(\nabla^{\ti{\mc{T}}})=\Hol(\nabla^{\ti{\mc{S}}})\subset\Spin(p+1,q+1).
\end{align}

\section{Fefferman-type constructions}\label{feff-const}
Let $\tilde{G}$ be a Lie group with Lie algebra $\frak{so}(p+1,q+1)$
and let $\tilde{P}\subset\tilde{G}$ be the stabiliser of a null-line $\ell\subset\mathbb{R}^{p+1,q+1}$.
Suppose we have an inclusion of Lie groups $i:G\hookrightarrow\tilde{G}$ with derivative $i:\g\to\tilde{\g}$. 
Assume  that the $G$-orbit $G\cdot o$ is open in $\tilde{G}/\tilde{P}$  and let  $P\subset G$ be a parabolic subgroup that contains the intersection $Q=G\cap \tilde{P}$. (In particular, this implies $\g/\p \cong \tilde{\g}/\tilde{\p}$ and $\tilde{\g}=\g+\tilde{\p}$).
This is the algebraic set up for Fefferman-type constructions as in \cite{cap-constructions}   inducing conformal structures of signature $(p,q)$.  

Since  Fefferman-type constructions have been studied quite intensively in the literature already, we recall the general construction here only briefly and refer to the literature  (e.g. \cite{cap-gover-cr-tractors} and \cite{cap-slovak-par}) for details.
Let $(\calG\to M,\omega)$ be a parabolic geometry of type $(G,P)$. One can form the correspondence space $\tilde{M}=\calG/Q=\calG\times_{P}P/Q$. 
The projection $\calG\to\tilde{M}$ is a $Q$-principal bundle, and from the defining properties of a Cartan connection one sees that $\omega\in\Omega^1(\mathcal{G},\g)$ is a Cartan connection also on  $\calG\to\tilde{M}$. So $(\mathcal{G}\to\tilde{M},\omega)$ is a Cartan geometry of type $(G,Q)$. As a next step, one considers the extended bundle $\tilde{\calG}=\calG\times_{Q}\tilde{P}$ with respect to the inclusion $Q\hookrightarrow \tilde{P}$.
 This is a principal bundle over $\tilde{M}$ with structure group $\tilde{P}.$
 Equivariant extension of $\omega$ yields a unique Cartan connection $\tilde{\omega}\in\Omega^1(\tilde{\calG},\tilde{\g})$ that restricts to $\omega$ on $\mathcal{G}$. Thus, one obtains a functor from parabolic geometries $(\calG\to M,\omega)$  of type $(G,P)$ to parabolic geometries $(\tilde{\calG}\to\tilde{M},\tilde{\omega})$ of type $(\tilde{G},\tilde{P})$.

\subsection{Normality}Next we 
derive a criterion suitable for our purposes that tells when this Fefferman-construction assigns a \emph{normal} conformal geometry $(\tilde{\calG},\tilde{\omega})$ to a regular, normal parabolic geometry $(\calG,\omega)$.
We will throughout  assume that the restriction of the Killing form $\tilde{B}$ of $\tilde{\g}$ to $\g$ is a non-zero multiple of the Killing form  $B$ of $\g$ (which is true for the inclusions  we are interested in).
We use $\tilde{B}$ to identify $(\g/\p)^*\cong\p_{+}$
and $(\tilde{\g}/\tilde{\p})^*\cong\tilde{\p}_{+}.$ Let $X_1,\cdots,X_n\in\g$ be elements inducing a basis of $\g/\p$ and extend these elements by $ X_{n+1},\cdots,X_{m}\in\p$ such that $X_1,\cdots, X_{m}$ induce a basis of 
$\g/\q\cong \tilde{\g}/\tilde{\p}.$
 Let $Z_1,\dots, Z_n$ be the dual basis of  $X_1,\cdots,X_n$ in $(\g/\p)^*\cong\p_{+}$ and $\tilde{Z}_1,\dots,\tilde{Z}_m$ be the dual basis of $X_1,\cdots, X_{m}$ in $(\tilde{\g}/\tilde{\p})^*\cong\tilde{\p}_{+}$. Then $\tilde{Z}_j-Z_j$ for $j=1,\dots,n$ are contained in the orthogonal complement $\g^{\perp}\subset\tilde{\g}$ with respect to the Killing form: For  $i=1,\dots,n,$ we have 
\begin{align*}
\tilde{B}(X_i,\tilde{Z}_j-Z_j)=\tilde{B}(X_i,\tilde{Z}_j)-\tilde{B}(X_i,Z_j)=\delta_{i,j}-\delta_{i,j}=0.
\end{align*}
For  $i=n+1,\dots,m,$ we have $\tilde{B}(X_i,\tilde{Z}_j)=0$ since $i\neq j$ and $\tilde{B}(X_i,Z_j)=0$ since 
$X_i\in \p$ and $Z_j\in\p_{+}$. Finally, we have $\tilde{B}(\q,\tilde{Z}_j)=0$ since $\q\subset \tilde{\p}$ and $\tilde{Z}_j\in\tilde{\p}_{+}$ and $\tilde{B}(\q,Z_j)=0$ since 
$\q\subset \p$ and $Z_j\in\p_{+}$.

Now suppose $\kappa:\calG\to\Lambda^2(\g/\p)^*\otimes\g$ is the curvature function of a normal parabolic geometry of type $(G,P)$. The normality condition reads
\begin{equation}\begin{aligned}
\partial^*(\kappa)(u)(X)=\partial^*_1(\kappa)(u)(X)&+\partial^*_2(\kappa)(u)(X)\\= 2\sum_{i=1}^{n}[\kappa(u)(X_i,X),Z_i]&+\sum_{i=1}^{n}\kappa(u)([X_i,[Z_i,X]])=0
\end{aligned}
\end{equation}
for all $u\in\calG$ and $X\in\g$.
Let $\tilde{\kappa}:\tilde{\calG}\to\Lambda^2(\tilde{\g}/\tilde{\p})^*\otimes\g$ be the curvature function of the associated conformal geometry. This geometry is normal if and only if 
\begin{align}
 \tilde{\partial}^*\tilde{\kappa}(\tilde{u})(\tilde{X})=2\sum_{i=1}^{m}[\tilde{\kappa}(\tilde{u})(X_i,\tilde{X}),\tilde{Z}_i]=0
\end{align}
for all $\tilde{u}\in\tilde{\calG}$ and $\tilde{X}\in\tilde{\g}$. By construction, we know that  $\tilde{\kappa}$ is a $\tilde{P}$-equivariant extension of $\kappa$  and elements of $\p$ insert trivially into $\tilde{\kappa}$. Since also $\tilde{\partial}^*$ is $\tilde{P}$-equivariant, to prove normality of $\tilde{\kappa}$ it suffices  to verify that 
\begin{align}\label{norm}
\tilde{\partial}^*\tilde{\kappa}(u)(X)= 2\sum_{i=1}^{n}[\tilde{\kappa}(u)(X_i,X),\tilde{Z}_i]=2\sum_{i=1}^{n}[\kappa(u)(X_i,X),\tilde{Z}_i]=0
\end{align}
for all for all $u\in\calG$ and $X\in\g$.

\begin{prop}\label{prop-norm}
Suppose that the parabolic geometry $(\mathcal{G},\omega)$ of type $(G,P)$ is regular and normal, the curvature function $\kappa$ takes values in $\Lambda^2(\g/\p)^*\otimes (\g\cap\tilde{\p})$ and the two summands in the normality condition vanish separately, i.e. $\partial^*_1(\kappa)=\partial^*_2(\kappa)=0$. Then $\tilde{\partial}^*(\tilde{\kappa})=0$, i.e. the induced conformal parabolic geometry is normal.
\end{prop}

\begin{proof}
Using that $\partial^*_1(\kappa)(u)(X)=2\sum_{i=1}^{n}[\kappa(u)(X_i,X),Z_i]=0$ and \eqref{norm},  we can rewrite $\tilde{\partial}^*\tilde{\kappa}(u)(X)$ as
 \begin{align}
2\sum_{i=1}^{n}[\kappa(u)(X_i,X),\tilde{Z}_i-Z_i].
 \end{align}
 We have observed that $\tilde{Z}_i-Z_i\in\g^{\perp}$ and by construction $\kappa(u)(X_i,X)\in\g$. Since the decomposition $\tilde{\g}=\g\oplus\g^{\perp}$ is invariant under the action of $\g$, this implies that 
 $\tilde{\partial}^*\tilde{\kappa}(u)(X)=\sum_{i=1}^{n}[\kappa(u)(X_i,X),\tilde{Z}_i-Z_i]\in \g^{\perp}$. On the other hand, since by assumption $\tilde{\kappa}(u)(X_i,X)\in\tilde{\p}$ and $\tilde{Z}_i\in\tilde{\p}_{+}$, we have $\tilde{\partial}^*\tilde{\kappa}(u)(X)\in\tilde{\p}_{+} $. 
 But  the intersection $\g^{\perp}\cap\tilde{\p}_{+}$ is zero:  Note that $\tilde{\p}_{+}=\tilde{\p}^{\perp}$, so any element in $\g^{\perp}\cap\tilde{\p}_{+}$ is orthogonal to $\g+\tilde{\p}=\tilde{\g}$. Since the Killing form is non-degenerate this implies $\g^{\perp}\cap\tilde{\p}_{+}={0}$ and we conclude that $\tilde{\partial}^*\tilde{\kappa}=0$.
\end{proof}

\begin{rema}
Suppose $\kappa$ is torsion-free, then Corollary 3.2 in \cite{cap-correspondence} shows that  it suffices to check that both $\partial^*_1$ and $\partial^*_2$ annihilate the harmonic curvature to conclude that they annihilate $\kappa$. 
If there is only one harmonic curvature component, then always $\partial^*_1(\kappa_{H})=\partial^*_2(\kappa_{H})=0.$ The reason for this is that the two summands $\partial^*_1(\kappa_{H})(u)(X)$ and $\partial^*_2(\kappa_{H})(u)(X)$ are contained in different grading components and cannot cancel.
\end{rema}

\section{From projective to conformal structures of signature $(n,n)$}\label{splitsig}
\subsection{The construction }\label{construction} For this construction
denote by $\De=\De_+^{n+1,n+1}\oplus\De_-^{n+1,n+1}$ the real $2^{n+1}$-dimensional
spin representation of $\tilde{G}=Spin(n+1,n+1)$. Then we fix two pure spinors $s_F\in\De_-^{n+1,n+1}, s_E\in\De_{\pm}^{n+1,n+1}$ with non-trivial pairing - here $s_E$\ lies
in $\De_+^{n+1,n+1}$\ if $n$ is even or $\De_-^{n+1,n+1}$ if $n$ is odd,
cf. \cite{baum-pseudospin}.
These assumptions guarantee 
that the kernels $E,F\subset\rr^{n+1,n+1}$\ of $s_E,s_F$\ with respect
to Clifford multiplication are complementary maximally isotropic subspaces.
Let now 
\begin{align*}G:=\{g\in \Spin(n+1,n+1): g\cdot s_E=s_E,g\cdot s_F=s_F\}\cong\SL(n+1),
\end{align*} and this defines an embedding
\begin{align*}
  \SL(n+1)\overset{i}{\embed}\Spin(n+1,n+1).
\end{align*}
Under $\SL(n+1)$ the space $\rr^{n+1,n+1}$ then decomposes into a copy
of the standard representation and the dual representation:
\begin{align}\label{decompR}
  \rr^{n+1,n+1}=E\oplus F=\rr^{n+1}\oplus{\rr^{n+1}}^*.
\end{align}
Note that this decomposition determines a $G$-invariant skew-symmetric involution $\mathbb{K}\in\Lambda^2\mathbb{R}^{n+1} $ acting by the identity on $E$ and minus the identity on $F$. In particular an embedding of
$\SL(n+1)$ can also be defined via such an involution.

We will realise $\Spin(n+1,n+1)$\ with respect
to the split signature form
\begin{align}\label{form(n,n)}
h= \begin{pmatrix}
                      0 & I_{n+1} \\
                      I_{n+1} & 0   \end{pmatrix},
\end{align}
such that the corresponding inclusion on the Lie algebra level is given by
\begin{align}
\notag \sl(n+1)&\embed \so(n+1,n+1)\\
A &\mapsto  \begin{pmatrix}
                      A& 0 \\
                      0 & -(A^t)  \end{pmatrix}.
\end{align}
Let $\tilde{P}\subset \tilde{G}$ be the stabiliser of the ray $\mathbb{R}_{+}\tilde{v}_{+}$ through the null-vector 
\begin{align}\label{tildev+}
\tilde{v}_{+}=\begin{pmatrix}
1 &
0&
\cdots&
0&
1
\end{pmatrix}^t\in\rr^{n+1,n+1}.
\end{align}
Then the group $Q:=i^{-1}(\tilde{P})\subset G$ consists of matrices of the form
\begin{align}
 \begin{pmatrix}
                      a&Z^t&b\\
                      0 & A &Y\\
                      0&0& a^{-1}  \end{pmatrix},
\end{align}
with $a\in \rr_{+}$, $b\in \rr$, $Z,Y\in\rr^{n-1}$ and $A\in \SL(n-1)$. This group $Q,$ which is not a parabolic subgroup, is contained in the parabolic subgroup $P\subset G,$ of the form
\begin{align}\label{P}
 \begin{pmatrix}
                      a&Z^t&b\\
                      0 & A &Y\\
                      0&X^t& c  \end{pmatrix},
\end{align}
defined as the stabilizer in $G$ of the ray $\mathbb{R}_{+} v_{+}$ through the vector 
\begin{align}
v_+=\begin{pmatrix}
1 &
0&
\cdots &
0
\end{pmatrix}^t\in\rr^{n+1}.
\end{align}
We denote by $\tilde{\g}, \tilde{\p}, \g, \p, \q$ the Lie algebras of the groups introduced above. Dimension count shows that the derivative $i':\g\to\tilde{\g}$ of  the inclusion $i:G\hookrightarrow\tilde{G}$ induces an isomorphism
$\g/\q\cong \tilde{\g}/\tilde{\p}.$ Hence the orbit $G\cdot o\subset \tilde{G}/\tilde{P}$ is open.
(But the action of $G$ on $\tilde{G}/\tilde{P}$ is not transitive; in addition to the open orbit  there are two lower dimensional orbits.)
That means that we can perform a Fefferman-type construction (as explained in \ref{feff-const}) from parabolic geometries of type $(G,P)$ on  to parabolic geometries of type $(\tilde{G},\tilde{P}).$
Since every parabolic geometry of type $(\tilde{G},\tilde{P})$ determines an underlying conformal spin structure (see e.g. \cite{cap-slovak-par}), this yields a construction of  a conformal spin structure on the correspondence space $\tilde{M}$ over a projective manifold $M$.

Let us describe the  correspondence space $\tilde{M}=\calG\times_{P} P/Q$ more carefully. Via the Cartan connection $\omega\in\Omega^1(\calG,\mathfrak{g})$,
the cotangent bundle $T^*M$ can be identified with $\calG\times_P(\mathfrak{g}/\mathfrak{p})^*$. Consider an arbitrary element $p\in P$; it is of the form
\begin{align}
p=\begin{pmatrix}a&Y^t\\
0&A
\end{pmatrix}
\end{align}
 for some $A\in \GL^{+}(n)$,
$a=(\mathrm{det} A)^{-1}$ and $Y\in\mathbb{R}^n$. The $P$-representation on $(\mathfrak{g}/\mathfrak{p})^*\cong\mathbb{R}^n$ is given by $\rho(p)(Y)=a (A^{-1})^{t} Y$.
If we  form the tensor product of this representation on $(\g/\p)^*$ with the 1-dimensional representation given by $\rho(p)=(\mathrm{det} A)^2=(a^{-1})^2$, then the resulting representation is $\rho(p)(Y)=a^{-1} (A^{-1})^{t} Y$. The corresponding representation space shall be denoted by $(\g/\p)^*[2]$. The action defined by this representation  is transitive on $(\mathfrak{g}/\mathfrak{p})^*[2]\backslash \{0\}$, and the isotropy subgroup of 
$e_n\in\mathbb{R}^{n}\cong(\mathfrak{g}/\mathfrak{p})^*[2]\backslash \{0\}$ is the group $Q$. Thus we may identify the correspondence 
space $\tilde{M}$ with $T^*M[2]\backslash\{0\}$.

\begin{prop}\label{(n,n)const}
 The Fefferman-type construction for the pairs of Lie groups $(G,P)$ and $(\tilde{G},\tilde{P})$ as above naturally associates a conformal spin structure of signature $(n,n)$ on $\tilde{M}=T^*M[ 2]\backslash \{0\}$ to an $n$-dimensional projective structure on $M$.
\end{prop}

\subsection{Induced structures on the conformal Fefferman space}\label{sec-induced}
Let $L=\mathbb{R}\tilde{v}_+$ be the line spanned by the null-vector $\tilde{v}_+$ and let  $L^{\perp}$ be the orthogonal complement with respect to $h$. Consider $\bar{E}=E\cap L^{\perp}$ and $\bar{F}= F\cap L^{\perp}$. 
The line $L$ is neither contained in $\bar{E}$ nor $\bar{F},$  and these two subspaces induce $n$-dimensional isotropic subspaces $\ee, \ef$ in $L^{\perp}/L$ with $1$-dimensional intersection $\mathrm{k}$.

We have a $\q$-invariant  identification $\g/\q\cong L^*\otimes L^{\perp}/L$ via $X\mapsto (\tilde{v}_+)^*\otimes X\cdot \tilde{v}_+$. Under this identification 
the subspace $f=\p/\q\subset \g/\q$ corresponds to $L^*\otimes \ef$. We denote by $e\subset\g/\q$ the subspace corresponding to $L^*\otimes \ee$, and then $e\cap f=\p'/\q\subset \g/\q$, where $\p'$ is the Lie algebra of $P'$ as in \eqref{P'}, which corresponds to $L^*\otimes \mathrm{k}$.
The $G$-invariant involution $\mathbb{K}\in\so(n+1,n+1)=\ti\g$\ defines
a $Q$-invariant element $k:=\mathbb{K}/\ti\p\in\ti\g/\ti\p\cong \g/\q$, which spans the $1$-dimensional intersection $e\cap f$.
The sum $e + f$ coincides with the orthogonal complement of $k$ in $\g/\q$.
Thus
\begin{align}
k\in e\cap f \subset k^{\perp}=e + f\subset \g/\q,
\end{align}
 both $e$\ and $f$\ are maximally isotropic (of dimension $n$) in $\g/\q,$ and in particular $k$ is null.
 
 It follows, that the tangent bundle $T\tilde{M}=\mathcal{G}\times_{Q}\g/\q$ has two  $n$-dimensional isotropic subbundles with one-dimensional intersection, corresponding to $e$ and $f$ and $e\cap f$. The bundle $\mathcal{G}\times_{Q}f$ is the vertical bundle for the projection $\tilde{M}\to M$.

The geometric tractor objects corresponding to the $G$-invariant algebraic data introduced in the beginning of \ref{construction}  will be denoted as follows:
The conformal standard tractor bundle $\tilde{\mathcal{T}}=\tilde{\mathcal{G}}\times_{\tilde{P}}\mathbb{R}^{n+1,n+1}=\mathcal{G}\times_{Q}\mathbb{R}^{n+1,n+1}$ naturally decomposes
as 
\begin{align}\label{decompT}
 \tilde{\mathcal{T}}=\mathcal{\tilde{E}}\oplus\mathcal{\tilde{F}}.
\end{align}
 The involution $\mathbb{K}$ gives rise to an adjoint tractor $\mathbf{K}\in\Gamma(\Lambda^2\tilde{\mathcal{T}})$ and the invariant spinors give rise to (pure) spin tractors $\mathbf{s}_{E}\in\Gamma(\tilde{\mathcal{S}}_{\pm})=\Gamma(\mathcal{G}\times_{Q}\Delta_{\pm})$ and $\mathbf{s}_{F}\in\Gamma(\tilde{\mathcal{S}}_{-})=\Gamma(\mathcal{G}\times_{Q}\Delta_{-})$ with non-trivial pairing, cf. \ref{construction}.

The conformal Cartan connection $\tilde{\omega}\in\Omega^1(\tilde{\mathcal{G}},\g)$ obtained via the Fefferman construction induces a tractor connection $\nabla^{\tilde{\mathcal{V}}}$ on each conformal tractor bundle $\tilde{\mathcal{V}}$. 
By construction, the decomposition of the tractor bundle \eqref{decompT} is
preserved by the induced conformal tractor connection and the adjoint tractor $\mb{K}$ and the spin tractors $\mb{s}_E, \mb{s}_F$ are all parallel with respect to the induced tractor connections on the respective bundles.
Note that we have not made any claims yet as to whether the  additional structure on the conformal tractor bundles is preserved by the normal conformal tractor connection $\nabla^{\tilde{\mathcal{V}},nor}$, which is a priori different from $\nabla^{\tilde{\mathcal{V}}}$.

\subsection{Relation between projective and conformal parallel tractors}\label{Relpartra} 
Suppose $\mathbb{V}$ is a $\tilde{G}$ representation, which is then also a $G$ representation, since $G\subset\tilde{G}$. Let  $\mathcal{V}=\mathcal{G}\times_{P}\mathbb{V}\to M$ be the associated projective tractor bundle and let $\tilde{\mathcal{V}}=\tilde{\mathcal{G}}\times_{\tilde{P}}\mathbb{V}=\mathcal{G}\times_{Q}\mathbb{V}\to \tilde{M}$ be the associated conformal tractor bundle.  
Let $\nabla^{\mathcal{V}}$ and $\nabla^{\tilde{\mathcal{V}}}$ be the tractor connections induced by $\omega$ and $\tilde{\omega}$. Sections of $\mathcal{V}$  bijectively correspond to $P$-equivariant functions $f:\mathcal{G}\to\mathbb{V}$, while sections  of $\tilde{\mathcal{V}}$ correspond to $Q$-equivariant functions $f:\mathcal{G}\to\mathbb{V}$. In particular, since $Q\subset P$, every section of $\mathcal{V}$ gives rise to a section of $\tilde{\mathcal{V}}$, and we can view $\Gamma(\mathcal{V})\subset\Gamma(\tilde{\mathcal{V}})$.

Conversely,   the proof of Proposition 3.3 in \cite{cap-gover-cr-tractors}  applied to our setting shows that  a section $\tilde{s}\in\Gamma(\tilde{\mathcal{V}})$ is contained in
$\in\Gamma(\mathcal{V})$ (i.e. the corresponding $Q$-equivariant function is actually $P$-equivariant) iff $\nabla^{\tilde{\mathcal{V}}}_{\xi}\tilde{s}=0$ for all $\xi$ in the vertical bundle of $\tilde{M}\to M$. The proof further shows that the tractor connection $\nabla^{\tilde{\mathcal{V}}}$ restricts to a connection on $\Gamma(\mathcal{V})\subset\Gamma(\tilde{\mathcal{V}})$, which coincides with $\nabla^{\mathcal{V}}$.
This implies  a bijective correspondence between $\nabla^{\tilde{\mathcal{V}}}$-parallel tractors in $\Gamma(\tilde{\mathcal{V}})$ and $\nabla^{\mathcal{V}}$-parallel tractors in $\Gamma(\mathcal{V})$.
If $\mathbb{V}$ is irreducible  as a $\tilde{G}$-representation but has a $G$-invariant subspace $\mathbb{W}\subset\mathbb{V},$ then this correspondence restricts to a bijective correspondence between parallel sections of $\mathcal{\tilde{W}}=\mathcal{G}\times_{Q}\mathbb{W}\to\tilde{M}$ and parallel sections of $\mathcal{W}=\mathcal{G}\times_{P}\mathbb{W}\to M$.

\subsection{Exceptional case: Dimension two}
In the special case of a projective structure in dimension $n=2$ the curvature function of a normal projective Cartan connection takes values in $\Lambda^2(\g/\p)^*\otimes\p_{+}$, see e.g. \cite{cap-slovak-par}. It is easily seen from the explicit matrices that $\p_{+}\subset\tilde{\p}\cap \g$. We can thus apply Proposition \ref{prop-norm} in this case, which shows:
\begin{prop}\label{prop-dim2}
 Suppose we are given a normal parabolic geometry $(\mathcal{G},\omega)$ encoding a two-dimensional projective structure. Then the associated conformal parabolic geometry $(\tilde{\mathcal{G}},\tilde{\omega})$ is normal, and thus   $\nabla^{\tilde{\mathcal{V}},nor}=\nabla^{\tilde{\mathcal{V}}}$ for any tractor bundle $\tilde{\mathcal{V}}$.
\end{prop}

 This has some immediate consequences (compare with the results in \cite{nurowski-sparling-cr}, \cite{dunajski-tod}, \cite{cap-gover-cr}): 

\begin{prop}\label{prop-22}
The split-signature conformal structures  obtained from two-dimensional projective structures  via the Fefferman-type construction  \ref{(n,n)const} have the following  properties:
\begin{enumerate}
 \item The conformal holonomy $\Hol(\nabla^{\tilde{\mathcal{T}},nor})$ is contained in $\SL(3).$
\item The normal conformal tractor connection $\nabla^{\tilde{\mathcal{T}},nor}$ preserves the decomposition $\tilde{\mathcal{T}}=\mathcal{\tilde{E}}\oplus\mathcal{\tilde{F}}.$
\item The adjoint tractor $\mathbf{K}$ is parallel with respect to the normal tractor connection, i.e. $\nabla^{\Lambda^2\tilde{\mathcal{T}},nor}\mathbf{K}=0$. Thus $\mathbf{K}$ corresponds to a normal conformal Killing field $k\in\mathfrak{X}(\tilde{M}),$ i.e. an infinitesimal conformal isometry that inserts trivially into Weyl-curvature and Cotton-tensor (cf. \cite{cap-gover-cr-tractors}).
\item The spin tractor bundle has two sections $\mathbf{s}_{E}$ and $\mathbf{s}_{F}$ with non-trivial pairing that are parallel with respect to the normal tractor connection, i.e. $\nabla^{\tilde{\mathcal{S}}_{+},nor}\mathbf{s}_{E}=0$ and $\nabla^{\tilde{\mathcal{S}}_{-},nor}\mathbf{s}_{F}=0$. Thus they correspond to two pure twistor spinors $\chi_e\in\Ga(S_{+}[\frac{1}{2}])$ and $\chi_f\in\Ga(S_-[\frac{1}{2}])$.

 \end{enumerate}
\end{prop}

\subsubsection{Almost Einstein structures}
A nice application of  the construction in dimension two is that it makes visible the properties of a projective structure that correspond to the existence of
 almost Einstein scales of the associated conformal structure.

\begin{prop}\label{prop-22ein}
 Suppose $(\tilde{M},[g])$ is a  conformal structure of signature $(2,2)$  associated to a $2$-dimensional projective structure $(M,[D])$ via the Fefferman-type construction. 
\begin{enumerate}[(1)]
\item Then $\mb{aEs}(\mc{C})=\mb{aRs}([D])\oplus\ker\Th_0^{T}$.
\item 
Let $g\in\mb{aEs}(\mc{C})$\ be defined on the open-dense subset
$U\subset M$\ and let $D^g$ be the Levi-Civita connection of $g$.
Then $g$ corresponds to an almost Ricci-flat structure of
$[D]$\ if and only if $D^g\chi_{f}=0$
and $g$\ corresponds to an element of $\ker\Th_0^{\mc{T}}$\ if
and only if $D^g\chi_{e}=0$.
In both cases it automatically follows that $\Ric(g)=0$.
\end{enumerate}
\end{prop}
\begin{proof}
\begin{enumerate}[(1)]
\item
We apply the relations between projective and conformal
parallel tractors discussed above in section \ref{Relpartra} to the conformal standard tractor bundle $\tilde{\mathcal{T}}=\tilde{\mathcal{G}}\times_{\tilde{P}}\mathbb{R}^{3,3}$. 
 As a $G=SL(3)$ representation $\mathbb{R}^{3,3}$ decomposes as $\mathbb{R}^{3,3}=\mathbb{R}^3\oplus{\mathbb{R}^3}^*,$ and thus  
the conformal standard tractor bundle $\tilde{\mathcal{T}}$ decomposes.  
 For a  $\nabla^{\tilde{\mathcal{T}}}=\nabla^{\tilde{\mathcal{E}}}+\nabla^{\tilde{\mathcal{F}}}$ parallel section $\tau=\tau_E+\tau_F$ of $\tilde{\mathcal{T}}=\tilde{\mathcal{E}}\oplus\tilde{\mathcal{F}}$, one summand corresponds to  a parallel section of the projective standard tractor bundle $\mathcal{E}=\mathcal{T}=\mathcal{G}\times_{P}\mathbb{R}^3$, and the other summand corresponds to a parallel section of the dual bundle $\mathcal{F}={\mathcal{T}}^*=\mathcal{G}\times_{P}{\mathbb{R}^3}^*$ (both equipped with the normal tractor connections $\nabla^{\mathcal{T}}$ and $\nabla^{{\mathcal{T}}^*}$). 
 
 Now by Proposition \ref{prop-dim2} we have $\nabla^{\tilde{\mathcal{T}}}=\nabla^{\tilde{\mathcal{T}},nor}$. It is well known that parallel conformal standard tractors for the normal tractor connection correspond to almost Einstein structures $\mb{aEs}(\mc{C})$, see \eqref{def-aEs}.  $\nabla^{{\mathcal{T}}^*}$-parallel projective co-tractors correspond to almost Ricci-flat structures $\mb{aRs}([D])$, see \eqref{aRsref},
and $\nabla^{\mathcal{T}}$-parallel projective standard tractors correspond to solutions of the projectively invariant differential operator $\Th_0^{\mathcal{T}},$ see \eqref{proj-standardBGG}.

\item
A parallel conformal standard tractor $s\in\Ga(\tilde{\mc{T}})$\
  corresponds to an almost Ricci-flat scale of $(M,[D])$ iff lies in $\Ga(\tilde{\mc{F}})$. On the other hand, parallel standard tractors  $\Ga(\ti{\mc{T}})$
  correspond to almost Einstein scales, so we have to characterise those
  $\si\in \mb{aEs}(\mc{C})$\ with $L_0^{\ti{\mc{T}}}\si\in\Ga(\ti{\mc{F}})$.
  Since $\ti{\mc{F}}$\ was defined as the kernel of $\mb{s}_B\in\Ga(\ti{\mc{S}})$\
  under Clifford multiplication, we equivalently have to check when
  \begin{align}
    \label{l0actchi}
    L_0^{\ti{\mc{T}}}(\si) \cdot \mb{s}_F=0.
  \end{align}
  Now $U=M\backslash\si^{-1}(\{0\})$ is open-dense in $M$, hence
  suffices by continuity to verify \eqref{l0actchi}\
  on that subset. On $U$\ we can use the Einstein metric $g$\ corresponding to the scale
  $\si$. Then according to \eqref{splitStd}
  \begin{align*}
    s=L_0^{\ti{\mc{T}}}(\si)=
    \begin{pmatrix}
    -\frac{1}{4}J \\
    0 \\
    1
    \end{pmatrix},
  \end{align*}
  where $J=g^{pq}\PP_{pq}$\ is the trace of the Schouten tensor.
  Using
  \begin{align*}
    \mb{s}_F=L_0^{\ti{\mc{S}}}(\chi_f)=
    \begin{pmatrix}
      \frac{1}{2\sqrt{2}}\crd\chi_f \\ 
      \chi_f
    \end{pmatrix}
  \end{align*}
  and the formula \eqref{traCli} for the tractor-Clifford action, equation \eqref{l0actchi} becomes
  \begin{align*}
    \begin{pmatrix}
      -\frac{1}{2\sqrt{2}} J \chi_f \\
      -\frac{1}{2}\crd\chi_f
    \end{pmatrix}
    =
  \begin{pmatrix}
    0 \\
    0 
  \end{pmatrix}.
  \end{align*}
  Since $\chi_f$\ satisfies
  the twistor equation $D_c\chi_f+\frac{1}{4}\crd\chi_f=0$,
  vanishing of $\crd\chi_f$\ implies that $\chi_f$\ is parallel
  with respect to $D$. In that case, since $\mb{s}_F$\ is
  $\na^{\ti{\mc{S}}}$-parallel,
  \begin{align*}
  \begin{pmatrix}
    0 \\
    0 
  \end{pmatrix}=
  \na^{\ti{\mc{S}}}_c \mb{s}_F=\na^{\ti{\mc{S}}}_c L_0^{\ti{\mc{S}}}\chi_f=
  \begin{pmatrix}
    \frac{1}{\sqrt{2}}\PP_{cp}\ga^p \\
    0
  \end{pmatrix},
  \end{align*}
  and thus, since $g$\ is Einstein, $J=0$\ and $\Ric(g)=0$.
  This shows that \eqref{l0actchi}\ holds
  for $\si\in\mb{aEs}(\mc{C})$\ if and only
  if $D\chi_f=0$ on $U$, and then $\Ric(g)=0$\ follows automatically.

  The discussion for the case where $g$\ corresponds to
  an element in $\ker\Th_0^{\mc{T}}$ is completely analogous.
\end{enumerate}
\end{proof}

\begin{rema}
  If $\mb{aRs}([D])\not =\{0\}$\ then $[D]$ is locally projectively flat,
and therefore also $\mc{C}$ is locally conformally flat:
For a projective $2$-dimensional structure the Weyl curvature as defined
in \eqref{formulaCprojective} always vanishes and
  the Cotton-tensor $A(D)$ as defined in \eqref{formulaAprojective} is projectively invariant and the
  complete obstruction against projective flatness. If $\hat D\in\mb{aRs}([D])$
  is a Ricci-flat affine connection on a open-dense subset $U\subset M$
  then $\Ric(\hat D)=0$\ implies that the Cotton-tensor $A(\hat D)$\ of $\hat D$ vanishes
  on $U$. If $D\in [D]$, then by projective invariance $A(D)=A(\hat D)=0$\ on
  $U$, and by continuity thus $A(D)=0$\ on all of $M$.
\end{rema}

\subsubsection{Conformal Killing fields}
Note that under $\mathfrak{sl}(3)$ the Lie algebra  $\mathfrak{so}(3,3)$ decomposes into the following irreducible pieces
\begin{align}
\mathbb{R}^3\oplus{\mathbb{R}^3}^*\oplus\mathfrak{sl}(3)\oplus\mathbb{R}.
\end{align}
Analogously to \cite{cap-gover-cr-tractors,mrh-sag-rank2} one can prove that:

\begin{prop}\label{prop-22kill}
The space of conformal Killing fields decomposes as
\begin{align}
 \mathbf{aEs}(\mc{C})\oplus\mathbf{inf}([D])\oplus\mathbb{R}k ,
\end{align} 
  where $k$ is the conformal Killing field from Proposition \ref{prop-22}, $ \mathbf{aEs}(\mc{C})$ a subspace isomorphic to the space of almost Einstein structures, i.e. solutions of \eqref{aEs-equ}, and $\mathbf{inf}([D])$ a subspace  isomorphic to the space of infinitesimal automorphisms of the original projective structure.
\end{prop}

\subsection{Remark: The construction for Lagrange contact structures}
Note that we can add an intermediate step to the construction of  section
\ref{construction}.
Let $P'$ be the parabolic in $G$ that stabilises the ray
$\mathbb{R}_{+}v_{+}$ and the $n$-dimensional subspace $\bar{E}$, i.e.
matrices of the form
\begin{align}\label{P'}
\begin{pmatrix}
                     a&Z^t&b\\
                     0 & A &Y\\
                     0&0& c  \end{pmatrix}.
\end{align}
Then obviously $Q\subset P'\subset P$. The correspondence space
$M'=\calG\times_{P} P/P'$ can be identified
with the projectivised cotangent bundle $\mathcal{P}(T^*M)$.
The parabolic geometry $(\calG\to M',\omega)$ of type $(G,P')$ defines a Lagrange contact structure on $\mathcal{P}(T^*M)$, i.e. a contact distribution $\mathcal{H}\subset TM'$ and a decomposition $\mathcal{H}=e' \oplus f'$ into two rank $n$ subbundles such that 
the restriction of the Levi bracket to $e'\times e'$ and $f'\times f'$ vanishes identically  (see e.g.
\cite{cap-slovak-par}).
Hence the construction of section \ref{construction}  can  be regarded as the composition of a correspondence space construction from projective to Lagrange
contact structures  with a Fefferman-type construction from Lagrange contact to conformal structures, which is similar to the original Fefferman construction;
one deals with different real forms of the same complex Lie groups in the two cases.

\begin{prop}
The Fefferman-type construction for Lagrange contact structures produces a normal conformal parabolic geometry iff
the parabolic geometry encoding the Lagrange contact structure is torsion-free.
\end{prop}

\begin{proof}
If the geometry is torsion-free, then there is only one non-trivial harmonic curvature component (cf. \cite{cap-slovak-par})
and $\partial^*_1$ and $\partial^*_2$ vanish separately on $\kappa_{H}$, and thus on $\kappa$.  
The harmonic curvature component $\kappa_{H}$ takes values in $\Lambda^2(\g/\p')^*\otimes({\g_{0}'}^{ss}\oplus\p_{+}')$ (see e.g. \cite{cap-zadnik-chains}). 
This is a $P'$ submodule, and so the the entire curvature takes values in that subspace. Since ${\g'_{0}}^{ss}\oplus\p'_{+}\subset\g\cap\tilde{\p}$ we can  apply Proposition \ref{prop-norm} to conclude normality.
 The converse direction is obvious since every normal conformal geometry is torsion-free and
$\g\cap\tilde{\p}\subset\p'.$ 
\end{proof}

Thus, as in the two dimensional case discussed before, we have:
\begin{cor}\label{2dimstructure}
For the split-signature conformal structures  coming from torsion-free Lagrange contact structures
\begin{enumerate}
 \item the conformal holonomy  is contained in $\SL(n+1),$
\item the normal conformal tractor connection $\nabla^{\tilde{\mathcal{T}},nor}$ preserves the decomposition $\tilde{\mathcal{T}}=\mathcal{\tilde{E}}\oplus\mathcal{\tilde{F}},$
\item the adjoint tractor $\mathbf{K}$ is parallel with respect to the normal conformal tractor connection and thus  it corresponds to a normal conformal Killing field,
\item the spin tractor bundle has two parallel sections $\mb{s}_E\in\Gamma(\ti{\mc{S}}_+)$\ and $\mb{s}_F\in\Gamma(\tilde{\mc{S}}_-)$\ with non-trivial pairing, and these correspond to two pure twistor spinors $\chi_e\in \Ga(S_{\pm}[\frac{1}{2}]),\chi_f\in \ga(S_-[\frac{1}{2}])$.

 \end{enumerate}
\end{cor}

\subsection{The projective construction for higher dimensions}
For $n>2$ the curvature of a normal projective Cartan connection is still contained in $\Lambda^2(\g/\p)^*\otimes\p$ but not in $\Lambda^2(\g/\p)^*\otimes\p_{+}$, and we cannot  invoke Proposition \ref{prop-norm} to conclude  that the induced conformal Cartan connection is normal. 

\begin{prop}\label{prop-nonnormal}
 For $n>2$ The conformal Cartan connection form $\ti\om\in\Om^1(\ti\G,\ti\g)$\ induced
by the normal projective Cartan connection form $\om\in\Om^1(\G,\g)$\
is normal if and only $\om$\ is flat, in which case also $\ti\om$ is flat.
\end{prop}

\begin{proof}
 If the induced conformal geometry is normal, then it is torsion-free, i.e. the curvature function $\ti\kappa$ takes values in $\Lambda^2(\tilde{\g}/\tilde{\p})^*\otimes(\tilde{\p}\cap\g)$.
But this is only possible if the harmonic curvature of the original projective geometry takes values in a $P$-submodule of  $\Lambda^2(\g/\p)^*\otimes\p/\p_{+}$ that is contained in $\Lambda^2(\g/\p)^*\otimes(\tilde{\p}\cap\g)/\p_{+}$,
and there is no such non-trivial $P$-invariant subspace. 
\end{proof}

\begin{rema}To relate this to the previous section, note: a Lagrange contact structure coming form a projective structure  via a correspondence space construction is torsion-free iff  it is flat, or equivalently, iff the projective structure is flat (see e.g. \cite{cap-slovak-par}).\end{rema}

 \subsubsection{Kostant codifferential of the curvature} In the non-flat case we need to understand how the normalised Cartan connection
form $\ti\om^{nor}$\ differs from $\ti\om$.  As a preliminary step for the normalisation to be carried out in the proof of Theorem \ref{prop-norm2},  we investigate the special form of
\begin{align*}\tilde{\partial}^*\tilde{\kappa}:\mathcal{G}\to(\tilde{\g}/\tilde{\p})^*\otimes \tilde{\p}\cong(\g/\q)^*\otimes \tilde{\p},
 \end{align*}
and 
\begin{align*}
\partial^*\tilde{\kappa}_0:\mathcal{G}\to(\g/\q)^*\otimes \tilde{\p}/\tilde{\p}_{+}
\end{align*}
(i.e. the composition of $\partial^*\tilde{\kappa}$ with the projection $\tilde{\p}\to\tilde{\p}/\tilde{\p}_{+}=\tilde{\g}_{0}$).
 
\begin{prop}\label{prop-delstarformnn}
 Suppose $\tilde{\omega}$ is  the conformal Cartan connection induced from a normal projective Cartan connection via the Fefferman-type construction.
Then, for any $u\in\mathcal{G}$,  $\tilde{\partial}^*\tilde{\kappa}(u)$ can be viewed as an element in 
\begin{align}\label{delkappa}
 f\otimes\Lambda^2\bar{F}
\end{align}
and  $\tilde{\partial}^*\tilde{\kappa}_0(u)$ defines an element in
\begin{align}\label{delkappa0}
f\otimes\Lambda^2 f.
\end{align}

\end{prop}
In this proof, we will often use the identification of $(\mathbb{R}^{n+1,n+1})^*$ with ${\mathbb{R}^{n+1,n+1}}$ provided by the bilinear form \eqref{form(n,n)}, which identifies $E^*$ with $F$.

\begin{proof}
A priori, $\tilde{\partial}^*\tilde{\kappa}(u)$ is an element of $(\g/\q)^*\otimes \tilde{\p}$. Since elements of $\p$ insert trivially into $\tilde{\kappa}(u)$ we have that $\partial^*\tilde{\kappa}(u)$ annihilates $f=\p/\q$,  and since $f$ is maximally isotropic  $\partial^*\tilde{\kappa}(u)$ can thus be viewed as an element in  $f\otimes\tilde{\p}$.

Next we determine the subspace of $\tilde{\p}\subset\Lambda^2\mathbb{R}^{n+1,n+1}$ where $\tilde{\partial}^*\tilde{\kappa}(u)(X)$ takes its values using \eqref{norm}. The space spanned by the elements $\tilde{Z}_i$, $i=1,...n,$ can be characterised as the annihilator of the vertical space $f=\p/\q,$ i.e. it is the space of all $\tilde{Z}\in\tilde{\p}_{+}$ such that $B(\tilde{Z},X)=0$ for all $X\in\p,$ where $B$ denotes the Killing form. One can easily see from the explicit form of $\p$ and $\tilde{v}_{+}$ (see \eqref{P} and \eqref{tildev+})
 that the image of the action of $\p$ on $\tilde{v}_+$ is $\bar{F}+L.$ 
  Furthermore, the action of an element $Z\in\tilde{\mathfrak{p}}_{+}$ annihilates $\tilde{v}_+$ and maps $X\cdot  \tilde{v}_+\in\bar{F}+L$ to (a multiple) of $B(Z,X)\tilde{v}_+$.  
Thus the subspace spanned by the $\tilde{Z}_i$, $i=1,...n,$ is contained the annihilator of $\bar{F}+L$ in $\tilde{\g}$.  
Note that  $\bar{F}+L$ is a $\p$ submodule and so is the annihilator of that subspace. Since $\tilde{\kappa}(u) (X,X_i)\subset\mathfrak{p}$ this implies that   $\tilde{\partial}^*\tilde{\kappa}(u)(X)$ annihilates $\bar{F}+L$.

 Now, the $\g$-module decomposition of $\tilde{\g}$ looks as follows
\begin{align*}
\ti{\g}=\Lambda^2(E\oplus F)=\underbrace{(E\otimes F)_0}_{\mathfrak{g}}\oplus \underbrace{(E\otimes F)_{Tr}\oplus \Lambda^2E\oplus\Lambda^2F}_{\mathfrak{g}^{\perp}}
\end{align*}
and in block-matrices
\begin{align*}
= \begin{pmatrix}
    E\otimes F& \Lambda^{2}E\\
    \Lambda^{2}F &E\otimes F 
  \end{pmatrix}.
\end{align*}
The assumption that the projective Cartan connection be normal  implies that $\tilde{\partial}^*\tilde{\kappa}(u)(X)\subset\g^{\perp}$, by \eqref{norm} and since $\tilde{Z}_i-Z_i\in\g^{\perp}.$
 Vanishing of $\tilde{\partial}^*\tilde{\kappa}(u)(X)$ on $\bar{F}$ and skew-symmetry implies that the $\Lambda^2E$-part of $\tilde{\partial}^*\tilde{\kappa}(u)(X)$ has to vanish. Vanishing on $\tilde{v}_+=\pi_E(\tilde{v}_+)+\pi_F(\tilde{v}_+)$ and on $\pi_F(\tilde{v}_+)\subset \bar{F}$ implies vanishing on $\pi_E(\tilde{v}_+)$. But then $\tilde{\partial}^*\tilde{\kappa}(u)(X)$ has also trivial $(E\otimes F)_{Tr}$-part and is indeed contained in the subspace of maps in $\Lambda^2F$ that vanish on $\pi_E(\tilde{v}_+),$ i.e. in $\Lambda^2\bar{F}$.
Then $\partial^*\tilde{\kappa}(u)(X)\subset\Lambda^2\bar{F}$ implies that  $\partial^*\tilde{\kappa}_0(u)(X)\subset \Lambda^2f$, which shows \eqref{delkappa0}.

\end{proof}

\begin{rema}\label{rem-lift}
We have seen in the proof of Proposition \ref{prop-norm} that $\tilde{\p}_{+}\cap\g^{\perp}=\{0\},$ and thus the restriction of the projection $\tilde{\p}\to\tilde{\p}/\tilde{\p}_{+}$ to the subspace $\tilde{\p}\cap\g^{\perp}$ is injective. Note that this implies that for every $\phi_0\in f\otimes \Lambda^2 f$ there is a unique element $\phi\in f\otimes\Lambda^2\bar{F}\subset f\otimes (\tilde{\p}\cap\g^{\perp})$ that projects onto $\phi_0$. 
\end{rema}

\subsubsection{Reduced Weyl-structures}

As a technical preliminary
to study how the normalised Cartan connection
form $\ti\om^{nor}$\ differs from $\ti\om$
we now relate the Weyl structures of the original Cartan
geometry $(\G,\om)$ and those of $(\ti\G,\ti\om)$:
\begin{prop}\label{prop-Weyl}
Any projective Weyl structure
\begin{align*}
  \G_0\overset{j}{\embed}\G
\end{align*}
induces a  conformal Weyl structure
\begin{align*}
  \G_0\embed\ti\G_0\overset{\ti{j}}{\embed}\ti \G.
\end{align*}
\end{prop}
\begin{proof}
\begin{align*}
  Q_0=G_0\cap Q=G_0\cap(G\cap\ti P)=G_0\cap \ti P=G_0\cap \ti G_0.
\end{align*}
We have $G_0\cong P/P_+$, and since $P_+\subset Q$, $Q_0\cong Q/P_+$,
and thus $G_0/Q_0\cong G/P$.
Therefore the reduction
$\G_0\overset{j}{\embed} \G$\ from $P$\ to $G_0$ over
the manifold $M$ induces a reduction from $Q\subset P\subset \ti P$\ to $Q_0\subset G_0\subset \ti Q$ over $\ti M$.
Composing the embedding $\G_0\overset{j}{\embed} \G$\ with
the natural embedding $\G\embed\ti \G$,
one obtains a reduction
\begin{align*}
  \G_0\overset{\ti{j}}{\embed}\ti \G
\end{align*}
from $\ti P$\ to $Q_0\subset\ti G_0$ over $\ti M$.
Let $\ti\G_0:=\G_0\times_{Q_0}\ti G_0$, then
the embedding $\G_0\embed \ti\G_0$\ is natural
and $\ti{j}$\ canonically extends to an embedding
of the $\ti G_0$-bundle $\ti\G_0$\ into the $\ti P$-bundle $\ti\G$,
which we simply denote by $\ti{j}$\ again.
We therefore
see that $\ti{j}$
is a reduced Weyl structure of the conformal Cartan bundle $\ti\G$.
\end{proof}

A version of this result in a more general context has been proved in \cite{jesse-quaternionic}.

\subsubsection{Preserved spin-tractors and induced twistor spinors}
Let $\mb{s}_F\in\Ga(\tilde{\mc{S}}_-)$ be the spin tractor with kernel $\ti{\mathcal{F}}\subset\tilde{\mathcal{T}}$ as in \ref{sec-induced}.
\def\t{\otimes}
\begin{thm}\label{prop-norm2}
$\mb{s}_F\in\Ga(\mc{S}_-)$
is parallel with respect to the normal conformal spin tractor
connection $\na^{\mc{S}_-,nor}\mb{s}_F=0.$
In particular, the conformal spin structure $(M,\mc{C})$\ carries
a canonical (pure) twistor spinor $\chi_f\in\Ga(S_-[\frac{1}{2}])$.
\end{thm}
\begin{proof}
  We are going to normalise the Cartan connection 
$\ti\om\in\Om^1(\ti\G,\ti\g)$\ 
 that is induced by the  projective Cartan connection $\om\in\Om^1(\G,\g)$.
  Any other conformal Cartan connection $\ti\om'$ differs
  from $\ti\om$\ by some $\Psi\in\Om^1(\ti\G,\ti\g)$: $\ti\om'=\ti\om+\Psi$.
  This $\Psi$\ must vanish on vertical fields and be $P$-equivariant.
  The condition on ${\ti\om}'$\ to induce the same conformal structure on $\ti M$\ as
  $\ti\om$ is that $\Psi$\ has values in $\ti\p\subset\ti\g$.
  One can therefore regard  $\Psi$\ as a $P$-equivariant
  function $\Psi:\ti\G\goesto (\ti\g/\ti\p)^*\t\ti\p$.

  The general theory of parabolic geometries, \cite{cap-slovak-par}, tells us that there is a unique such $\Psi$\ such that the curvature function $\ti \kappa'$\ 
  of $\ti\om'$ satisfies $\ti\delstar \ti \ka'=0$, and then
  $\ti\om'$\ is the normal conformal Cartan connection $\ti\om^N$.

  The normalisation of $\ti\om$\ proceeds by homogeneity of $(\ti\g/\ti\p)^*\t\ti\p,$
which decomposes into two homogeneous components according to the decomposition 
$\ti\p=\ti\g_0\oplus\ti\p_{+}$. The failure
  of $\ti\om$\ to be 
  normal is $\ti\delstar(\ti \ka):\G\goesto (\ti\g/\ti\p)^*\t\ti\p$.
  In the first step of normalisation one looks for a ${\Psi}^{0}$ such
  that ${\ti\om}^{1}=\ti\om+{\Psi}^{0}$\ has $\ti\delstar\ti\ka'$\ taking
  values in the highest homogeneity $\ti\delstar\ti\ka':\ti\G\goesto (\ti\g/\ti\p)^*\t\ti\p_+$.

  To write down this first normalisation it is useful
  to employ a Weyl structure $\ti\G_0\overset{j}{\embed}\ti\G$,
  and by Proposition \ref{prop-Weyl}\ we can take a
  Weyl structure that is induced by a $Q_0$-reduction 
  \begin{align*}
    \G_0\overset{j}{\embed}\G\embed\ti \G.
  \end{align*}

  This allows us to project $\ti\delstar\ti \ka$\ to $(\ti\delstar\ti\ka)_{0}:\G_0\goesto (\ti\g/\ti\p)^*\t\ti\g_0$
  and to employ the $\ti G_0$-equivariant Kostant Laplacian
  $\ti\lapl:(\ti\g/\ti\p)^*\t\ti\g_0\goesto (\ti\g/\ti\p)^*\t\ti\g_0$.
  For the first normalisation step we need to form a map
  ${\Psi}^{0}:\ti\G\goesto (\ti\g/\ti\p)^*\t\ti\p$ that agrees
  with $-\ti\lapl^{-1}(\ti\delstar\ti\ka)_{0}$ in the $\ti\g_0$-component.
  If we have formed any such ${\Psi}^{0}$\ along $\G_0\overset{j}{\embed}\ti\G$
  we can just equivariantly extend this to all of $\ti\G$.

  Now $\ti\lapl$ restricts to an invertible endomorphism
  of $((\ti\g/\ti\p)^*\t\ti\g_0)\cap\im\ti\delstar$ that
  acts by scalar multiplication on each of the three
  $\ti G_0$-irreducible components of $((\ti\g/\ti\p)^*\t\ti\g_0)\cap\im\ti\delstar$.
  To write down this decomposition we use that under $\ti G_0$ one
  has
    $\ti\g/\ti\p\cong\rr^{n,n}.$
  As a $\ti G_0$-module, $({\rr^{n,n}}^{*}\otimes\ti\g_0)\cap \im\ti\delstar$\ decomposes
  into 
  \begin{align*}
    (({\rr^{n,n}}^{*}\otimes\ti\g_0)\cap \im\ti\delstar)_{tr}\oplus (({\rr^{n,n}}^{*}\otimes\ti\g_0)\cap \im\ti\delstar)_{alt}\oplus(({\rr^{n,n}}^{*}\otimes\ti\g_0)\cap \im\ti\delstar)_{\odot},
  \end{align*}
  where
  \begin{align*}
    &(({\rr^{n,n}}^{*}\otimes\ti\g_0)\cap \im\ti\delstar)_{tr}=\rr^{n,n}, \\
    &(({\rr^{n,n}}^{*}\otimes\ti\g_0)\cap \im\ti\delstar)_{alt}=\La^3\rr^{n,n}
  \end{align*}
  and
  $(({\rr^{n,n}}^{*}\otimes\ti\g_0)\cap \im\ti\delstar)_{\odot}$ the highest weight
  component, which is the trace- and alternation-free part.
  
  Now $\rr^{n,n}$\ has the $Q_0$-invariant subspace $f$,
  and, it was shown in Proposition \ref{prop-delstarformnn} that
  \begin{align*}
    \ti\delstar\ti\ka_{0}\in f\t\La^2 f.
  \end{align*}
  This shows that $\ti\delstar\ti\ka_{0}$\ has
  no trace-component, and since $\La^3 f\subset f\t\La^2 f$,
  we have that 
  \begin{align*}
    &(\ti\delstar\ti\ka_{\ti\g_0})_{tr}=0,\\
    &(\ti\delstar\ti\ka_{\ti\g_0})_{alt}\in f\t\La^2 f\ \mr{and}\\
    &(\ti\delstar\ti\ka_{\ti\g_0})_{\odot}\in f\t\La^2 f.
  \end{align*}
  Since $\ti\lapl$\ preserves these components it follows that also 
  \begin{align*}
    -\ti\lapl^{-1}\ti\delstar\ti\ka_{0}\in f\t\La^2 f.
  \end{align*}
  For each element in $f\t\La^2 f\subset (\ti\g/\ti\p)^*\t\ti\g_0$
  there exists a unique element in $(\ti\g/\ti\p)^*\t\La^2 \bar{F}\subset (\ti\g/\ti\p)^*\t\ti \p$\ with that $\ti\g_0$-component, cf. Remark \ref{rem-lift}.
  This defines a canonical
  \begin{align*}
    \Psi^0:\G_0\goesto (\ti\g/\ti\p)^*\t \La^2 \bar{F}\subset (\ti\g/\ti\p)^*\t\ti\p
  \end{align*}
  for the first normalisation step and we set ${\ti\om}^{1}=\ti\om+{\Psi}^{0}$.
  
  Since $F\subset \rr^{n+1,n+1}$\ is the kernel of the pure spinor
  $s_F\in\De_-^{n+1,n+1}$ we see that the tractor spinor $\mb{s}_F$
  induced by the constant map
  \begin{align*}
    \G\goesto\De_-^{n+1,n+1},\ 
    u\mapsto s_F
  \end{align*}
  is still parallel with respect to ${\ti\om}^{1}$.

  Now we want to see that also after the second normalisation
  step, which yields the normal conformal Cartan connection
  ${\ti\om}^{2}=\ti\om^{nor}$, the tractor spinor $\mb{s}_F$\ is still
  parallel.
  One has $\ti\om^{nor}={\ti\om}^{1}+{\Psi}^{1}$,
  with ${\Psi}^{1}:\G_0\goesto\ti\p_+$, and
  we denote the spin tractor connections on $\ti{\mc{S}}_-$\ induced by
  ${\ti\om}^{1}$\ and $\ti\om^{nor}$
  by $\na^{\ti{\mc{S}}_-,1}$ resp. $\na^{\ti{\mc{S}}_-,nor}$.

  Recall from \ref{sec-spintrac}\ that
  $[\ti{\mc{S}}_{-}]_g=
  \begin{pmatrix}
    S_{+}[-\frac{1}{2}] \\
    S_{-}[\frac{1}{2}]
  \end{pmatrix}$.
Now $\na^{\ti{\mc{S}}_-,1}\mb{s}_F=0$\ and
  \begin{align*}
    (\na^{\ti{\mc{S}}_-,nor}\mb{s}_F-\na^{\ti{\mc{S}}_-,1}\mb{s}_F)={\Psi}^{1}\mb{s}_F\in\Ga(S_+[-\frac{1}{2}])
    \subset\Ga(\ti{\mc{S}}_-),
  \end{align*}
  and therefore
  \begin{align}\label{nanormbs}
    \na^{\ti{\mc{S}}_-,nor}\mb{s}_F\in\Ga(S_+[-\frac{1}{2}]).
  \end{align}
  Let $\mb{s}_F=
  \begin{pmatrix}
    \tau \\ \chi_f
  \end{pmatrix}\in\Ga(\ti{\mc{S}}_-)$.
  Then \eqref{nanormbs}\ says explicitly that
  \begin{align*}
    \begin{pmatrix}
      D_c\tau+\frac{1}{\sqrt{2}}\PP_{cp}\ga^p\chi_f\\
      D_c\chi_f+\frac{1}{\sqrt{2}}\ga_c\tau
    \end{pmatrix}
=
\begin{pmatrix}
  * \\ 0
\end{pmatrix}.
  \end{align*}
  It follows in particular that $\chi_f$\ is a twistor spinor,
  and necessarily $\tau=\frac{1{\sqrt{2}}n}\crd\chi_f$ - which
  shows that $\mb{s}_F=L_0^{\ti{\mc{S}}_-}(\chi_f)$.
  But then $D_c\tau+\frac{1}{\sqrt{2}}\PP_{cp}\ga^p\chi_f=0$\ is a
  differential consequence of that equation, and
  thus indeed 
  \begin{align*}
    0=\na^{\ti{\mc{S}}_-,nor}(L_0^{\ti{\mc{S}}_-}\chi_f)=\na^{\ti{\mc{S}}_-,nor}\mb{s}_F.
  \end{align*}
\end{proof}

Since $\Hol(\mc{C})=\Hol(\na^{\ti{\mc{S}},nor})$ Proposition \ref{prop-norm2} in particular implies that the induced conformal structures have reduced holonomy:
\begin{cor}
The conformal holonomy $\Hol(\mc{C})$\ is 
contained in
the isotropy subgroup of $s_{F}\in\De_-^{n+1,n+1}$ in $Spin(n+1,n+1)$; this is
$SL(n+1)\ltimes {\Lambda}^{2}(\mathbb{R}^{n+1})^{*}\subset Spin(n+1,n+1)$.
\end{cor}

\end{document}